\documentclass[12pt, reqno]{amsart}
\usepackage{amsmath, amsthm, amscd, amsfonts, amssymb, graphicx, color}
\usepackage[bookmarksnumbered, colorlinks, plainpages]{hyperref}

\newtheorem{theorem}{Theorem}[section]
\newtheorem{lemma}[theorem]{Lemma}
\newtheorem{proposition}[theorem]{Proposition}
\newtheorem{corollary}[theorem]{Corollary}
\theoremstyle{definition}
\newtheorem{definition}[theorem]{Definition}
\newtheorem{example}[theorem]{Example}

\theoremstyle{remark}
\newtheorem{remark}[theorem]{Remark}
\numberwithin{equation}{section}

\begin{document}
\setcounter{page}{1}

\title[Orthogonality and ordered Banach spaces]{Orthogonality in $\ell _p$-spaces and its bearing on ordered Banach spaces}

\author[A. K. Karn]{Anil Kumar Karn$^1$}

\address{$^{1}$ School of Mathematical Science, National Institute of Science Education and Research, Institute of Physics Campus, 
Sachivalaya Marg, P. O. Sainik School, Bhubaneswar- 751005, India.}
\email{\textcolor[rgb]{0.00,0.00,0.84}{anilkarn@niser.ac.in; anil.karn@gmail.com}}



\begin{abstract}
We introduce a notion of $p$-orthogonality in a general Banach space $1 \le p \le \infty$. 
We use this concept to characterize $\ell _p$-spaces among Banach spaces and also among complete 
order smooth $p$-normed spaces. We further introduce a notion of $p$-orthogonal decomposition in 
order smooth $p$-normed spaces. We prove that if the $\infty$-orthogonal decomposition holds in 
an order smooth $\infty$-normed space, then the $1$-orthogonal decomposition holds in the dual 
 space. We also give an example to show that the above said decomposition may not be unique.
\end{abstract} 
\maketitle

\section{Introduction}

Let $H$ be a real Hilbert space. For $x, y \in H$ we have $x \perp y$ ($x$ is orthogonal to $y$) 
if and only if $\langle x, y \rangle = 0$. Note that $x \perp y$ if and only if $\Vert x + ky 
\Vert ^2 = \Vert x \Vert ^2 + \Vert ky \Vert ^2$ for all $k \in \mathbb{R}$. This idea easily 
extends to $\ell_p$-spaces and more generally to all normed linear spaces for $1 \le p \le \infty$. 
Let us recall there are several notions of orthogonality in literature \cite{GB, RJ, BR, IS}. In 
this paper, we consider a special case to suit our model.

We introduce the notion of $p$-orthogonality in a general Banach space $1 \le p \le \infty$ 
(definitions given below). We use this concept to characterize $\ell _p$-spaces among Banach spaces. We further note that this concept has a natural bearing on the order structure in order smooth $p$-normed spaces. As a consequence, we specialize the above said characterization of 
$\ell -p$-spaces among complete order smooth $p$-normed spaces.

The notion of $p$-orthogonality  fits well with the decomposition of self-adjoint elements as 
differences of positive elements in a $C^{\ast}$-algebra or its dual \cite{KR, GP}. In this paper, 
we generalize to order smooth $\infty$-normed spaces as a duality result. More precisely, we prove 
that if the $\infty$-orthogonal decomposition holds in an order smooth $\infty$-normed space, then 
the $1$-orthogonal decomposition holds in the dual space.

We end this paper with an example to show that the above said decomposition may not be unique and 
hence needs further exploitations to characterize the decomposition in $C^{\ast}$-algebras.

\section{Orthogonality in $\ell _p$-spaces}

Let $1 \le p \le \infty$ and let $X$ be a Banach space. For $x, y \in X$ we say that $x$ is 
$p$-orthogonal to $y$, ( $x \perp _p y$ ), if 
$$ \Vert x + ky \Vert ^p = \Vert x \Vert ^p + \Vert ky \Vert ^p, \quad 1 \le p < \infty$$ 
and
$$\Vert x + ky \Vert = \max \{ \Vert x \Vert , \Vert ky \Vert \} , \quad p = \infty$$ 
for all $k \in \mathbb{R}$.

Further, we say that $\perp _p$ is additive on $X$, if $x \perp _p y$ and $x \perp _p z$ implies 
$x \perp _p (y+z)$. Note that in this case $x^{\perp _p} = \{ y \in X: x \perp _p y \}$ is a 
subspace of $X$.

A subset $S$ of $X$ is called $p$-orthogonal if $0 \not\in S$ and $x \perp _p y$ if $x, y \in S$ 
with $x \not= y$. If, in addition, $\Vert x \Vert = 1$ for all $x \in S$, we say that $S$ is a 
$p$- orthonormal set in $X$. We say that $S$ is total if the linear span of $S$ is dense in $X$. 

Let us note that a $p$-orthogonal set $U$ in $X$ is linearly independent. In fact, for $x_1, 
\dots , x_n \in U$ and $\alpha _1, \dots , \alpha _n \in \mathbb{R}$, we have
$$\Vert \sum _{k=1}^n \alpha _k x_k \Vert ^p = \sum _{k=1}^n \Vert \alpha _k x_k \Vert ^p$$ 
when $1 \le p < \infty$ and for $p = \infty$
$$\Vert \sum _{k=1}^n \alpha _k x_k \Vert = \max _{k=1}^n \Vert \alpha _k x_k \Vert .$$ 
Now, it is routine to prove the following characterization of $\ell _p$-spaces.

\begin{theorem}
Let $1 \le p \le \infty$. If $E = \{ e_i : i \in I \}$ is the standard basis for $\ell _p(I)$. 
Then $\perp _p$ is additive in $X$ and $E$ is a total $p$-orthonormal set in $\ell _p(I)$. 
Conversely, let $\perp _p$ be additive in a Banach space $X$ and let  $U$ is a total 
$p$-orthonormal set in $X$. Then $X$ is isometrically isomorphic to $\ell _p(U)$. For 
$p = \infty$, we replace $\ell _p$ by $c_0$.
\end{theorem}

In fact, it is clear from the observation made before theorem 2.1 that the linear span of $U$ is 
isometrically isomorphic to $c_{00}(U)_p$. This observation also leads us to a natural order 
structure (on a subspace) that corresponds to a $p$-orthogonal set $U$ in a normed linear 
space $X$ (when $\perp _p$ is additive on $X$). We just need to recall the appropriate definitions 
and result \cite{AK}.
\begin{definition}
Let $\left( V, V^+ \right)$ be a real ordered vector space such that $V^+$ is proper and 
generating and let $\Vert~\Vert$ be a norm on $V$ such that $V^+$ is $\Vert~\Vert$-closed. For a 
fixed real number $p$, $1 \le p < \infty$, consider the following conditions on $V$:
\begin{enumerate}
\item[(O.p.1)] For $u, v, w \in V$ with $u \le v \le w$, we have 
$\Vert v \Vert \le (\Vert u \Vert ^p + \Vert w \Vert ^p)^{1/p}$.
\item[(O.p.2)] For $v \in V$ and $\epsilon >0$, there are $u_1, u_2 \in V^+$ 
such that $v = u_1 - u_2$ and 
$(\Vert u_1 \Vert ^p + \Vert u_2 \Vert ^p)^{1/p} \le \Vert v \Vert + \epsilon $.

For $p = \infty$, further consider the following conditions on $V$:

\item[(O.$\infty$.1)] For $u, v, w \in V$ with $u \le v \le w$, we have 
$\Vert v \Vert \le \max \{ \Vert u \Vert , \Vert w \Vert \}$.
\item[(O.$\infty$.2)] For $v \in V$ and $\epsilon >0$, there are $u_1, u_2 \in V^+$ 
such that $v = u_1 - u_2$ and 
$\max \{ \Vert u_1 \Vert , \Vert u_2 \Vert \} \le \Vert v \Vert + \epsilon $.
\end{enumerate}
\end{definition}
\begin{theorem}
Let $\left( V, V^+ \right)$ be a real ordered vector space such that $V^+$ is proper 
and generating and let $\Vert \Vert$ be a norm on $V$ such that $V^+$ is 
$\Vert \Vert$-closed. For a fixed $p$, $1 \le p \le \infty$, we have
\begin{enumerate}
\item $\Vert~\Vert$ satisfies the condition $O.p.1$ on $V$ if and only if 
$\Vert~\Vert ^{\prime}$ satisfies the condition $O.p^{\prime}.2$ on $V^{\prime}$.
\item $\Vert~\Vert$ satisfies the condition $O.p.2$ on $V$ iv and only if 
$\Vert~\Vert ^{\prime}$ satisfies the condition $O.p^{\prime}.1$ on $V^{\prime}$.
\end{enumerate}
\end{theorem}
\begin{definition}
Let $\left( V, V^+ \right)$ be a real ordered vector space such that $V^+$ is proper 
and generating and let $\Vert \Vert$ be a norm on $V$ such that $V^+$ is 
$\Vert \Vert$-closed. For a fixed $p$, $1 \le p \le \infty$, we say that $V$ is an
order smooth $p$-normed space, if $\Vert \Vert$ satisfies conditions $O.p.1$ and
$O.p.2$ on $V$.
\end{definition}
\begin{theorem}
Let $\left( V, V^+ \right)$ be a real ordered vector space such that $V^+$ is proper and 
generating and let $\Vert \Vert$ be a norm on $V$ such that $V^+$ is $\Vert \Vert$-closed. For a 
fixed $p$, $1 \le p \le \infty$, $V$ is an order smooth $p$-normed space if and only if its 
Banach dual $V^{\prime}$ is an order smooth $p^{\prime}$-normed space satisfying the condition 
\begin{enumerate}
\item[(OS.$p^{\prime}$.2)] For $f \in V^{\prime}$, there are $g_1, g_2 \in V^{\prime +}$ 
such that $\Vert g_1 \Vert ^{p^{\prime}} + \Vert g_2 \Vert ^{p^{\prime}} = 
\Vert v \Vert^{p^{\prime}}$ and $f = g_1 - g_2$.\\
(A similar modification in the case of $p = 1$.)
\end{enumerate}
\end{theorem}
It is easy to note that (real) $\ell _p$-spaces are order smooth $p$-normed spaces. In this light the following result is straight forward.
\begin{proposition}
Let $U$ be a $p$-orthogonal set in a normed linear space $X$ in which $\perp _p$ is additive and 
let $\langle U \rangle$ denote the linear span of $U$. Set $\langle U \rangle ^+ = \{ 
\sum_{k=1}^n \alpha _k x_k : x_k \in U, \alpha \ge 0; k \in \mathbb{N} \}$. Then $\left( 
\langle U \rangle , \langle U \rangle ^+ , \Vert \Vert \right)$ is an order smooth $p$-normed 
space.
\end{proposition}

We can also give the following order theoretic characterization for $\ell_p$-spaces.

\begin{theorem}
 Let $1 \le p < \infty$ and let $V$ be a (norm) complete order smooth $p$-normed space. If 
 $\perp_p$ is additive on $V^+$ and $U$ is a total orthonormal set in $V^+$, then $V$ is 
 isometrically order isomorphic to $\ell_p(U)$.
\end{theorem}

\begin{proof}
The main stay of the proof is to show that
$$\langle U \rangle^+ := \langle U \rangle \cap V^+ = \{ \sum_{u \in U_1} \alpha_u u : \alpha_u 
\ge 0 \hskip 2mm \textrm{for all} \hskip 2mm U_1; U_1 \hskip2mm \textrm{a finite subset of} 
\hskip 2mm U \} .$$ 
We prove this using the following two lemma.
\begin{lemma}
Let $u_1, u_2 \in V^+$ with $u_1 \perp_p u_2$. If $u_1 - u_2 \in V^+$, then $u_2 = 0$.
\end{lemma}
\begin{proof}
Put $u_1 - u_2 = u$. then $0 \le u \le u_1$ so that 
$$\Vert u_1 \Vert^p \ge \Vert u \Vert^p = \Vert u_1 - u_2 \Vert^p = \Vert u_1 \Vert^p + \Vert 
u_2 \Vert^p .$$
Now, it follows that $u_2 = 0$.
\end{proof}
\begin{lemma}
 Let $u_1, \dots , u_n \in U$ be distinct. Then $\sum_{i=1}^n \alpha_i x_i \in V^+$ if and only 
 if $\alpha_i \ge 0$ for all $i= 1, \dots , n$.
\end{lemma}
\begin{proof}
If $\alpha_i \ge 0$ for all $i= 1, \dots , n$, then $\sum_{i=1}^n \alpha_i x_i \in V^+$. Conversely, let $\sum_{i=1}^n \alpha_i x_i \in V^+$.
If $\alpha_i <0$ for some $i$, then the set 
$$I_1 = \{ i : 1 \le i \le n \hskip 2mm \textrm{and} \hskip 2mm \alpha_i < 0 \} \not= \emptyset .$$
Now, $\sum_{i=1}^n \alpha_i x_i = \sum_{i \not\in I_1} \alpha_i x_i - \sum_{i \in I_1} (-\alpha_i) x_i \in V^+$ with 
$\sum_{i \in I_1} (-\alpha_i) x_i$ and $\sum_{i \not\in I_1} \alpha_i x_i$ in $V^+$ so that by 
Lemma 2.7, $\sum_{i \in I_1} (-\alpha_i) x_i = 0$. Thus
$$ 0 = \Vert \sum_{i \in I_1} (-\alpha_i) x_i \Vert^p = \sum_{i \in I_1} (-\alpha_i)^p.$$
This leads to a contradiction: $\alpha_i = 0$ for all $i \in I_1$. Thus the result holds.
\end{proof}
\noindent
{\it Proof of Theorem 2.6:} By Lemma 2.8, 
$$\langle U \rangle^+ = \{ \sum_{u \in U_1} \alpha_u u : \alpha_u \ge 0 \hskip 2mm 
\textrm{for all} \hskip 2mm U_1; U_1 \hskip2mm \textrm{a finite subset of} \hskip 2mm U \} .$$ 
Now it routine to show that $\langle U \rangle$ is isometrically order isomorphic to $c_{00}(U)_p$ 
in a natural way. Also $\overline{\langle U \rangle} = V$ is isometrically isomorphic to 
$\ell_p(U)$. Thus it only remains to show that $V^+ = \overline{\langle U \rangle^+}$.

Let $w \in V^+$. Then there is $(\alpha_u)_{u \in U} \in \ell_p(U)$ such that $w = \sum_{u \in U} 
\alpha_uu$. Further we may find a countable set $\{ u_n\}$ in $U$ such that $\alpha_u = 0$ if 
$u \not\in \{ u_n: n \in \mathbb{N} \}$. Let us write $\alpha_{u_n} = \alpha_n$ so that 
$w = \sum_{n\in \mathbb{N}}\alpha_n u_n$. Next, put
$$I_1 = \{ n \in \mathbb{N}: \alpha_n > 0 \} \hskip 2mm \textrm{and} \hskip 2mm I_2 = \{ n \in 
\mathbb{N}: \alpha_n < 0 \} .$$
Since $w \in V^+$, $I_1 \not= \emptyset$. Further, if $I_2 = \emptyset$, then $w \in 
\overline{\langle U \rangle^+}$ and we are done. Thus 
assume to the contrary that $I_2 \not= \emptyset$. Arrange $I_1$ and $I_2$ as increasing 
sequences $\{ n_k \}$ and $\{ m_k\}$ respectively. Note that these sequences may be finite too. 
Put $w_{1,t} = \sum_{k=1}^t \alpha_{n_k} u_{n_k}$ and $w_{2,t} = \sum_{k=1}^t (-\alpha_{m_k}) 
 u_{m_k}$. Further put $w_1 = \sum_{k=1}^{\infty} \alpha_{n_k} u_{n_k}$ and $w_2 = 
\sum_{k=1}^{\infty} (-\alpha_{m_k}) u_{m_k}$. Then $w_{1,t} \to w_1$ and $w_{2,t} \to w_2$ as 
$t \to \infty$. Also $w_1, w_2 \in \overline{\langle U \rangle^+}$ with $w = w_1 - w_2$. We show 
that $w_1 \perp_p w_2$. For this let $\lambda \in \mathbb{R}$. Then
\begin{align*}
\Vert w_1 + \lambda w_2 \Vert^p &= \lim_{t \to \infty} \Vert w_{1,t} + \lambda w_{2,t} \Vert^p\\
&= \lim_{t \to \infty} \left( \Vert w_{1,t} \Vert^p + \Vert \lambda w_{2,t} \Vert^p \right)\\
&= \Vert w_1 \Vert + \Vert \lambda w_2 \Vert^p.
\end{align*}
It follows that $w_1 \perp_p w_2$. Now by Lemma 2.7, $w_2 = 0$ which leads to a contradiction that $I_2 = \emptyset$. This completes the proof.
\end{proof}
\begin{remark}
 If we can prove a counterpart of Lemma 2.7 for $p = \infty$, then Lemma 2.8 and consequently Theorem 2.6 also hold for $p = \infty$ where
$\ell_p$ is replaced by $c_0$. This result we shall obtain as a corollary of another result where we characterize $\infty$-orthogonality
in order smooth $\infty$-normed spaces.
\end{remark}

\section{Orthogonality in order smooth $\infty$-normed spaces}

\begin{definition}
 Let $V$ be an order smooth $p$-normed space, $1 \le p \le \infty$. For $v \in V \setminus \{ 0 \}$ we say that $f \in V^{\prime}$ supports
$v$ if $\Vert f \vert = 1$ and $\Vert v \Vert = f(v)$. The set of all supports of $v$ will be denoted by $Supp(v)$. For $u \in V^+ \setminus
\{ 0 \}$, we write, $Supp_+(u)$ for $Supp(u) \cap V^{\prime +}$.
\end{definition}
By Hahn-Banach theorem, $Supp(v) \not= \emptyset$, if $v \in V \setminus \{ 0 \}$. Moreover, it is weak*-compact and convex too.

\begin{proposition}
  Let $V$ be an order smooth $p$-normed space, $1 \le p \le \infty$. For $u \in V^+ \setminus \{ 0 \}$, $Supp_+(u) \not= \emptyset$.
\end{proposition}
\begin{proof}
 First let $1 < p \le \infty$ (so that $1 \le p^{\prime} < \infty$). Find $f \in Supp(u)$. Since $V^{\prime}$ satisfies $(OS.p^{\prime}.2)$,
we can find $f_1, f_2 \in V^{\prime +}$ such that $f = f_1 - f_2$ and $1 = \Vert f 
\Vert^{p^{\prime}} = \Vert f_1 \Vert^{p^{\prime}} + \Vert f_2 \Vert^{p^{\prime}}$. Thus
$$\Vert u \Vert = f(u) = f_1(u) - f_2(u) \le f_1(u) \le \Vert f_1 \Vert \Vert u \Vert .$$
As $u \not= 0$, we get that $\Vert f_1 \Vert \ge 1$ so that $\Vert f_1 \Vert = 1$ and $f_2 = 0$. Thus $f_1 = f \in Supp_+(u)$.

Next let $p = 1$. Find $f \in Supp(u)$. Since $V^{\prime}$ satisfies $(OS.{\infty}.2)$, we can 
find $f_1, f_2 \in V^{\prime +}$ such that $f = f_1 - f_2$ and $1 = \Vert f \Vert = \max 
(\Vert f_1 \Vert, \Vert f_2 \Vert )$. As above we have
$$\Vert u \Vert = f(u) = f_1(u) - f_2(u) \le f_1(u) \le \Vert f_1 \Vert \Vert u \Vert$$
so that $\Vert f_1 \Vert = 1$ and $f_2(u) = 0$. Thus $\Vert u \Vert = f(u) = f_1(u)$ and consequently, $f_1 \in Supp_+(u)$.
\end{proof}
In the next result we characterize $\infty$-orthogonality in order smooth $\infty$-normed spaces.
Now onwards, in this section $V$ denotes an order smooth $\infty$-normed space unless otherwise 
stated.
\begin{theorem}
 Let $V$ be an order smooth $\infty$-normed space. Suppose that $u_1, u_2 \in V^+ \setminus \{ 0 \}$ and let $W$ be the linear span of
$u_1, u_2$. Then the following statements are equivalent:
\begin{enumerate}
 \item $\Vert \Vert u_1 \Vert^{-1} u_1 + \Vert u_2 \Vert^{-1} u_2 \Vert = 1$;
 \item $u_1 \perp_{\infty} u_2$;
 \item For $f_i \in Supp_+(u_i)$, $i = 1, 2$, we have $g_1 \perp_1 g_2$ with $g_i(u_j) = 0$ if $i \not= j$ where $g_i = f_i\vert_W$, 
$i = 1, 2$.
\end{enumerate}
\end{theorem}
\begin{proof}
Without any loss of generality, we may assume that $\Vert u_i \Vert = 1$, $i = 1, 2$. First, let $\Vert u_1 + u_2 \Vert = 1$. Let $\lambda
> 0$.

Case.1: $\lambda \le 1$.

In this case,
$$\Vert u_1 + \lambda u_2 \Vert = \Vert \lambda (u_1 + u_2) + (1 - \lambda ) u_1 \Vert \le \lambda \Vert u_1 + u_2 \Vert + (1- \lambda)
\Vert u_1 \Vert = 1.$$
Case.2: $\lambda > 1$.

In this case,
$$\Vert u_1 + \lambda u_2 \Vert = \Vert (u_1 + u_2) + (\lambda - 1 ) u_2 \Vert \le \Vert u_1 + 
u_2 \Vert + (\lambda - 1) \Vert u_2 \Vert = \lambda .$$
Thus, in either case, $\Vert u_1 + \lambda u_2 \Vert \le \max (1, \lambda )$. Further, as 
$\lambda > 0$ we also have
$$\max (1, \lambda ) = \max (\Vert u_1 \Vert , \Vert \lambda u_2 \Vert ) \le \Vert u_1 + 
\lambda u_2 \Vert .$$
Thus for $\lambda \ge 0$, we get $\Vert u_1 + \lambda u_2 \Vert = \max ( \Vert u_1 \Vert , 
\Vert \lambda u_2 \Vert )$.

Again, for $\lambda > 0$, $- \lambda u_2 \le u_1 - \lambda u_2 \le u_1$ so that 
$$\Vert u_1 - \lambda u_2 \Vert \le \max ( \Vert u_1 \Vert , \Vert \lambda u_2 \Vert ) .$$
Next, let $f_i \in Supp_+(u_i)$, $i = 1, 2$. Then $f_i(u_i) = 1 = \Vert f_i \Vert$, $i = 1, 2$. 
As $\Vert u_1 + u_2 \Vert = 1$ we get
$$1 = \Vert u_1 + u_2 \Vert \ge f_1 ( u_1 + u_2 ) = 1 + f_1 (u_2) \ge 1 .$$
Thus $f_1 (u_2) = 0$. Dually, $f_2(u_1) = 0$. Now it follows that
$$\Vert u_1 - \lambda u_2 \Vert \ge \vert f_1 (u_1 - \lambda u_2 ) = 1$$
and
$$\Vert u_1 - \lambda u_2 \Vert \ge \vert f_2 (u_1 - \lambda u_2 ) = \lambda$$
so that
$$\Vert u_1 - \lambda u_2 \Vert \ge \max (1, \lambda ) = \max (\Vert u_1 \Vert , \Vert \lambda u_2 \Vert ) .$$
Therefore, $u_1 \perp_{\infty} u_2$.

Next, let $u_1 \perp_{\infty} u_2$. Let $f_i \in Supp_+(u_i)$, $i = 1, 2$ and put $g_i = f_i\vert_W$, $i = 1, 2$. Then as above, 
we have $g_i(u_j) = 0$ if $i \not= j$. For $\alpha_1, \alpha_2 \in \mathbb{R}$, we have
\begin{align*}
\Vert \alpha_1 g_1 + \alpha_2 g_2 \Vert &= \sup \{ \vert ( \alpha_1 g_1 + \alpha_2 g_2 )
(\lambda_1 u_1 + \lambda_2 u_2 )\vert : \Vert \lambda_1 u_1 + \lambda_2 u_2 \Vert \le 1 \}\\
&= \sup \{ \vert \alpha_1 \lambda_1 + \alpha_2 \lambda_2 \vert : \max (\vert \lambda_1 \vert ,
\vert \lambda_2 \vert ) \le 1 \}\\
&= \vert \alpha_1 \vert + \vert \alpha_2 \vert .
\end{align*}
In particular, $\Vert g_i \Vert = 1$, $i = 1, 2$ so that $g_1 \perp_1 g_2$.

Finally assume that for $f_i \in Supp_+(u_i)$, $i = 1, 2$, we have $g_1 \perp_1 g_2$ with $g_i(u_j) = 0$ if $i \not= j$ where 
$g_i = f_i\vert_W$, $i = 1, 2$. It is easy to note that $\Vert g_i \Vert = 1$, $i = 1, 2$. Thus
\begin{align*}
\Vert u_1 + u_2 \Vert &= \sup \{ \vert ( \alpha_1 g_1 + \alpha_2 g_2 ) (u_1 + u_2 )\vert :
\Vert \alpha_1 g_1 + \alpha_2 g_2 \Vert \le 1 \}\\
&= \sup \{ \vert \alpha_1 + \alpha_2 \vert : \vert \alpha_1 \vert + \vert \alpha_2 \vert \le 1 \} 
= 1.
\end{align*}
This completes the proof.
\end{proof}
\begin{remark} \hfill

\begin{enumerate}
 \item In the above theorem, we have $f_1 \perp_1 f_2$. In fact, for $\alpha_1, \alpha_2 \in \mathbb{R}$
$$\vert \alpha_1 \vert + \vert \alpha_2 \vert \ge \Vert \alpha_1 f_1 + \alpha_2 f_2 \Vert \ge \Vert \alpha_1 g_1 + \alpha_2 g_2 \Vert \\
= \vert \alpha_1 \vert + \vert \alpha_2 \vert .$$
In particular, $\alpha_1 f_1 + \alpha_2 f_2$ is a norm preserving extension of $\alpha_1 g_1 + \alpha_2 g_2$ for all $\alpha_1, \alpha_2 \in 
\mathbb{R}$.
\item Let $u_1, u_2 \in V^+\setminus \{ 0 \}$. Since $- u_2 \le u_1 - u_2 \le u_1$ and $u_i \le u_1 + u_2$ for $i = 1, 2$, we get that
$\Vert u_1 - u_2 \Vert \le \max ( \Vert u_1 \Vert , \Vert u_2 \Vert ) \le \Vert u_1 + u_2 \Vert$. Thus by the above theorem, we may conclude 
that $u_1 \perp_{\infty} u_2$ if and only if $\Vert \Vert u_1 \Vert^{-1} u_1 + \Vert u_2 \Vert^{-1} u_2 \Vert = \Vert \Vert u_1 \Vert^{-1} 
u_1 - \Vert u_2 \Vert^{-1} u_2 \Vert$.
\end{enumerate}

\end{remark}
Now we can present Lemma 2.7 for $p = \infty$.
\begin{corollary}
 Let $V$ be an order smooth $\infty$-normed space and let $u_1, u_2 \in V^+$ with $u_1 
 \perp_{\infty} u_2$. If $u_1 - u_2 \in V^+$, then $u_2 = 0$.
\end{corollary}
\begin{proof}
 Assume to the contrary that $u_2 \not= 0$. Clearly, $u_1 \not= 0$. If $f \in Supp_+(u_2)$, then by Theorem 3.3, $f(u_1) = 0$ and $f(u_2) = 
\Vert u_2 \Vert$. Now, as $u_1 - u_2 \in V^+$ we have $0 \le f(u_1 - u_2) = - \vert u_2 \Vert$. But, then $u_2 = 0$ which contradicts the 
assumption that $u_2 \not= 0$. Thus the result holds.
\end{proof}
Now the counterparts of Lemma 2.8 and Theorem 2.6 for $p = \infty$ follow as a routine matter.
\begin{theorem}
 Let $V$ be a (norm) complete order smooth $\infty$-normed space in which $\perp_{\infty}$ is additive. If $V^+$ contains a total
$\infty$-orthonormal set $U$, then $V$ is isometrically order isomorphic to $c_0(U)$.
\end{theorem}
Now we consider some special cases.
\begin{definition}
Let $V$ be a (real) ordered space. An increasing net $\left\{e_\lambda\right\}$ in
$V^+$ is called an {\it approximate order unit} for $V$ if for each $v \in V$ there
is  $k >0$ such that $k e_{\lambda} \pm v \in V^+$ for some $\lambda$.
\end{definition}
In this case $\{ e_{\lambda} \}$  determines a seminorm $\Vert~\Vert_a$  on $V$ that satisfies 
$(O.{\infty}.1)$ and $(O.{\infty}.2)$. We call $\big( V, \{ e_{\lambda} \}\big)$ an {\it 
approximate order unit space}, if $\Vert~\Vert_a$ is a norm on $V$ in which $V^+$ is closed. Thus
an approximate order unit space is an order smooth $\infty$-normed space. It may be noted that
the self-adjoint part of a $C^{\ast}$-algebra is an approximate order unit space.

When $e_{\lambda} = e$ for all $\lambda$ we drop the term "approximate" in the above notions. For 
example, $(V, e)$ denotes an order unit space.   For details, please refer to \cite{KV1, KV2, KV3}.
\begin{corollary}
Let $(V, e)$ be an order unit space. Then for $u_1, u_2 \in V^+ \setminus \{ 0 \}$, we have 
$u_1 \perp_{\infty} u_2$ if and only if $\Vert u_1 \Vert^{-1} u_1 + \Vert u_2 \Vert^{-1} u_2 \le 
e$.
\end{corollary}
\begin{proof}
Without any loss of generality, we may assume that $\Vert u_i \Vert = 1$, $i = 1, 2$. First, let
$u_1 \perp_{\infty} u_2$. Then $\Vert u_1 + u_2 \Vert = 1$ so that $u_1 + u_2 \le e$. Conversely, 
let $u_1 + u_2 \le e$. Then $\Vert u_1 + u_2 \Vert \le 1$. Also, $0 \le u_1 \le u_1 + u_2$ so 
that $1 = \Vert u_1 \Vert \le \Vert u_1 + u_2 \Vert$. Thus by Theorem 3.3, $u_1 \perp_{\infty} 
u_2$.
\end{proof}
We can generalize this result to approximate order unit spaces by recalling the fact that 
For $u \in V^+$, we have $\Vert u \Vert \le 1$ if and only if for each $\epsilon > 0$ there is 
$\lambda$ such that $u \le (1 + \epsilon ) e_{\lambda}$.
\begin{corollary}
Let $\big( V, \{ e_{\lambda} \}\big)$ be an order unit space. Then for $u_1, u_2 \in V^+ 
\setminus \{ 0 \}$, we have $u_1 \perp_{\infty} u_2$ if and only if for each $\epsilon > 0$, 
there is $\lambda$ such that $\Vert u_1 \Vert^{-1} u_1 + \Vert u_2 \Vert^{-1} u_2 \le (1 + 
\epsilon) e_{\lambda}$.
\end{corollary}
Let us say that $u \in V^+ \setminus \{ 0 \}$ has an $\infty$-orthogonal pair in $V^+ \setminus 
\{ 0 \}$ if there exists a $v \in V^+ \setminus \{ 0 \}$ such that $u \perp_{\infty} v$. Note that 
in this case, there exists $f \in V^{\prime +}$ with $\Vert f \Vert = 1$ such that $f(u) = 0$. 
In this case we say that $f$ is a $Crust$ of $u$. the set of all crusts of $u$ will be denoted by 
$Crust_+(u)$.
\begin{corollary}
Let $(V, e)$ be an order unit space and let $u \in V^+ \setminus \{ 0 \}$. Then $u$ has an 
$\infty$-orthogonal pair in $V^+ \setminus \{ 0 \}$ if and only if $Crust_+(u) \not= \emptyset$.
\end{corollary}
\begin{proof}
By the above observation, if $u$ has an $\infty$-orthogonal pair in $V^+ \setminus \{ 0 \}$, then 
$Crust_+(u) \not= \emptyset$. Conversely let $Crust_+(u) \not= \emptyset$. Without any loss of 
generality we may assume that $\Vert u \Vert = 1$ so that $0 \le u \le e$ and that $0 \le e - u 
\le e$. Let $f \in Crust_+(u)$. Then $f$ is a state of $V$ with $f(u) =0$. Thus $f(e) = 1$ so that 
$f(e - u) = 1$. It follows that $\Vert e - u \Vert = 1$. Now, by Corollary 3.8, $e-u$ is a 
non-zero $\infty$-orthogonal pair of $u$.
\end{proof}
\begin{remark}
Let $u \in V^+$ with $\Vert u \Vert = 1$ where $V$ is an order unit space. If $v \in V^+$ with 
$\Vert v \Vert = 1$ is such that $u \perp_{\infty} v$, then $v \le e -u$ by Corollary 3.8. Also,
by Corollary 3.10, $u \perp_{\infty} (e - u)$. Thus $e - u$ is the greatest element of $V^+$ of 
norm one that is $\infty$-orthogonal to $u$, if $Crust_+(u) \not= \emptyset$.
\end{remark}

\section{Orthogonality in order smooth $1$-normed spaces}

In the following, we characterize $\perp_1$-orthogonality in order smooth $1$-normed spaces. The 
result is somewhat dual to Theorem 3.3.
\begin{theorem}
  Let $V$ be an order smooth $1$-normed space, $u_1, u_2 \in V^+ \setminus \{ 0 \}$ and $f_i \in Supp_+(u_i)$, $i = 1, 2$. let $W$ be the 
linear span of $u_1, u_2$ and put $g_i = f_i\vert_W$, $i = 1, 2$.. Then the following statements are equivalent:
\begin{enumerate}
 \item $u_1 \perp_1 u_2$ and $g_1(u_2) = 0 = g_2(u_1)$;
 \item $g_1 \perp_{\infty} g_2$. 
\end{enumerate}
\end{theorem}
\begin{proof}
 Without any loss on generality we may assume that $\Vert u_i \Vert = 1$, $i = 1, 2$. Further note that $\Vert g_i \Vert = 1$, $i = 1, 2$. 

First assume that $u_1 \perp_1 u_2$ and $g_1(u_2) = 0 = g_2(u_1)$. Then
\begin{align*}
\Vert g_1 + g_2 \Vert &= \sup \{ \vert ( g_1 + g_2 ) (\lambda_1 u_1 + \lambda_2 u_2 )\vert :
\Vert \lambda_1 u_1 + \lambda_2 u_2 \Vert \le 1 \}\\
&= \sup \{ \vert \lambda_1 + \lambda_2 \vert : \vert \lambda_1 \vert + \vert \lambda_2 \vert \le 1 \} = 1.
\end{align*}
Thus by Theorem 3.3, $g_1 \perp_{\infty} g_2$.

Conversely, assume that $g_1 \perp_{\infty} g_2$. Since $\Vert u_i \Vert = 1$ we have $g_i(u_i) = 1$, $i = 1, 2$. Thus for $\lambda_1, 
\lambda_2 > 0$
\begin{align*}
\lambda_1 + \lambda_2 &= \lambda_1 g_1(u_1) + \lambda_2 g_2(u_2)\\
&\le (g_1 + g_2)(\lambda_1 u_1 + \lambda_2 u_2)\\
&\le \Vert g_1 + g_2 \Vert \Vert \lambda_1 u_1 + \lambda_2 u_2 \Vert\\
&\le \lambda_1 + \lambda_2.
\end{align*}
It follows that $g_1(u_2) = 0 = g_2(u_1)$ and that $\Vert \lambda_1 u_1 + \lambda_2 u_2 \Vert = 
\Vert \lambda_1 u_1 \Vert + \Vert \lambda_2 u_2 \Vert$ for all $\lambda_1, \lambda_2 > 0$. Again, 
for $\lambda_1, \lambda_2 > 0$, we have 
\begin{align*}
\lambda_1 + \lambda_2 &= \lambda_1 g_1(u_1) + \lambda_2 g_2(u_2)\\
&= (g_1 - g_2)(\lambda_1 u_1 - \lambda_2 u_2)\\ 
&\le \Vert g_1 - g_2 \Vert \Vert \lambda_1 u_1 - \lambda_2 u_2 \Vert\\ 
&\le \lambda_1 + \lambda_2.
\end{align*}
Thus $\Vert \lambda_1 u_1 + \lambda_2 u_2 \Vert = \Vert \lambda_1 u_1 \Vert + \Vert \lambda_2 
u_2 \Vert$ for all $\lambda_1, \lambda_2 \in \mathbb{R}$ so that $u_1 \perp_1 u_2$.
\end{proof}
\begin{remark}
Under the assumptions of the above theorem, $g_1 \perp_1 g_2$ whenever $f_1 \perp_1 f_2$. In fact,
as $g_1 + g_2 = (f_1 + f_2)\vert_W$ we have $\Vert g_1 + g_2 \Vert \le \Vert (f_1 + f_2) \Vert = 
1$. Further, as $0 \le g_1 \le g_1 + g_2$ we have $1 = \Vert g_1 \Vert \le \Vert g_1 + g_2 \Vert$.
Thus by Theorem 3.3, $g_1 \perp_1 g_2$. However, we are not sure about the converse.
\end{remark}
Next, we specialize $1$-orthogonality in base normed spaces. First let us recall the definition. 
Let $V$ be a positively generated real ordered vector space. A subset $B$ of $V^+$ is called a 
base for $V^+$, if $B$ is convex, $0 \not\in B$ and for any $u \in V^+\setminus \{ 0 \}$ there exists a unique $k > 0$ and $b \in B$ such that $u = kb$. In this case $B$ determine a seminorm 
$\Vert~\Vert_b$ on $V$ which satisfies $(O.1.1)$ and $(O.1.2)$. We say that $(V, B)$ is a base 
normed space if $\Vert~\Vert_b$ is a norm and $V^+$ is closed in it. Thus a base normed space is 
an order smooth $1$-normed space \cite{KV2, KV3, AK}.
\begin{lemma}
Let $(V, B)$ be a base normed space and let $u_1, u_2 \in V^+\setminus \{ 0 \}$. Suppose that 
$f_i \in Supp_+(u_i)$, $i = 1, 2$. If $f_i(u_j) = 0$ for $i \not= j$, then $u_1 \perp_1 u_2$.
\end{lemma}
\begin{proof}
Since a base norm is additive on $V^+$, for $\lambda_1, \lambda_2 >0$, we have 
$$\Vert \lambda_1 
u_1 + \lambda_2 u_2 \Vert = \Vert \lambda_1 u_1 \Vert + \Vert \lambda_2 u_2 \Vert .$$ 
Also in general 
$$\Vert \lambda_1 u_1 - \lambda_2 u_2 \Vert \le \Vert \lambda_1 u_1 \Vert + \Vert \lambda_2 u_2 
\Vert$$ 
if $\lambda_1, \lambda_2 >0$.

Since $V$ is a base normed space, $V^{\prime}$ is an order unit space. Thus as $\Vert f_i \Vert 
= 1$, $i = 1, 2$, we get 
$$\Vert f_1 - f_2 \Vert \le \max ( \Vert f_1 \Vert , \Vert f_2 \Vert ) \le 1.$$ 
Also $f_i(u_i) = \Vert u_i \Vert$ for $i = 1, 2$ and $f_i(u_j)  = 0$ if $i \not= j$ so that for 
$\lambda_1, \lambda_2 >0$
\begin{align*}
\Vert \lambda_1 u_1 - \lambda_2 u_2 \Vert &\ge (f_1 - f_2)(\lambda_1 u_1 - \lambda_2 u_2)\\ 
&= \lambda_1 f_1(u_1) + \lambda_2 f_2(u_2)\\
&= \Vert \lambda_1 u_1 \Vert + \Vert \lambda_2 u_2 \Vert .
\end{align*}
Hence $u_1 \perp_1 u_2$.
\end{proof}
As an application, now we prove a duality between a decomposition of a (self-adjoint) element (mapping) as a difference of two positive elements (mappings).
\begin{theorem}
Let $V$ be an order smooth $\infty$-normed space. Assume that for each $v \in V$ there are $v_1,
v_2 \in V^+$ such that $v_1 \perp_{\infty} v_2$ and $v = v_1 - v_2$. Then for each $f \in 
V^{\prime}$ there are $f_1, f_2 \in V^{\prime +}$ such that $f_1 \perp_1 f_2$ and $f = f_1 - f_2$. 
\end{theorem}
\begin{proof}
Let $f \in V^{\prime}$. Without any loss of generality we may assume that $f \not\in V^{\prime +} 
\cap (- V^{\prime +})$. Since $V$ be an order smooth $\infty$-normed space, $V^{\prime}$ is an 
order smooth $1$-normed space satisfying $(OS.1.1)$. Thus there are $f_1, f_2 \in V^{\prime +} 
\setminus \{ 0 \}$ such that $f = f_1 - f_2$ and $\Vert f \Vert = \Vert f_1 \Vert + \Vert f_2 
\Vert$. Find $v \in V$ with $\vert v \Vert = 1$ such that $\Vert f \Vert = f(v)$. By assumption,
there are $v_1, v_2 \in V^+$ such that $v_1 \perp_{\infty} v_2$ and $v = v_1 - v_2$. Thus
$1 = \Vert v \Vert = \max \{ \Vert v_1 \Vert , \Vert v_2 \Vert \}$. Now
\begin{align*}
\Vert f \Vert = f(v) &= (f_1 - f_2)(v_1 - v_2)\\
&= \left( f_1(v_1) + f_2(v_2) \right) - \left( f_1(v_2) + f_2(v_1) \right)\\
&\le \left( f_1(v_1) + f_2(v_2) \right)\\
&\le \Vert f_1 \Vert \Vert v_1 \Vert + \Vert f_2 \Vert v_2 \Vert\\
&\le \Vert f_1 \Vert + \Vert f_2 \Vert = \Vert f \Vert .
\end{align*}
\noindent so that $f_1(v_2) = 0 = f_2(v_1)$, $\Vert v_1 \Vert = 1 = \Vert v_2 \Vert$ and 
$\Vert f_i \Vert = f_i(v_i)$, $i = 1, 2$. Thus it follows from Theorem 4.1 and Remark 4.2 that 
$f_1 \perp_1 f_2$.
\end{proof}
\begin{remark}

We are not able to prove a converse of this decomposition. We suspect that this may attribute to 
not uniqueness of such decomposition as shown in the next example. We expect some additional 
conditions in this regard. 
\end{remark}
\begin{example}
Consider $f(x) = \cos x \in C [0, 2\pi ]$. Then $f = f^+ - f^-$ with $f^+ \perp_{\infty} f^-$.
Next, let $g_1 (x) = \cos ^2 (\frac{1}{2} x)$ and $g_2 (x) = \sin ^2 (\frac{1}{2} x)$. Again, then 
$f = g_1 - g_2$ with $g_1 \perp_{\infty} g_2$.
\end{example}

\bibliographystyle{amsplain}

\end{document}